\documentclass[12pt]{amsart}  %[12pt] while developing, omitted for final
\usepackage[latin1]{inputenc}
\usepackage{amsmath, bm, mathdots, mathrsfs} 
\usepackage{amsfonts}
\usepackage{amssymb}
\usepackage{stmaryrd}
\usepackage{latexsym} 
\usepackage{graphicx}
\usepackage{subfigure}
\usepackage{xcolor}
\usepackage{hyperref}
\usepackage{verbatim}
\usepackage[all]{xy}
\usepackage{graphics}
\usepackage{pdfsync}
\usepackage{tikz}
\usepackage{url}
\usepackage{enumerate}
\usepackage{ytableau}

\DeclareFontFamily{U}{rcjhbltx}{}
\DeclareFontShape{U}{rcjhbltx}{m}{n}{<->rcjhbltx}{}
\DeclareSymbolFont{hebrewletters}{U}{rcjhbltx}{m}{n}

\DeclareMathSymbol{\shin}{\mathord}{hebrewletters}{152}

\newcommand{\ME}[1]{\textcolor{black}{#1}}
\newcommand{\JL}[1]{\textcolor{black}{#1}}
\newcommand{\svw}[1]{\textcolor{black}{#1}}
\newcommand{\finv}[1]{\textcolor{black}{#1}}
\newcommand{\refy}[1]{\textcolor{black}{#1}}
\newcommand{\refz}[1]{\textcolor{black}{#1}}

\theoremstyle{definition}
\newtheorem{theorem}{Theorem}[section]
\newtheorem*{theorem*}{Theorem}
\newtheorem{proposition}[theorem]{Proposition}

\newtheorem{lemma}[theorem]{Lemma}
\newtheorem{corollary}[theorem]{Corollary}

\newtheorem{definition}[theorem]{Definition}
\newtheorem{example}[theorem]{Example}

\theoremstyle{remark}

\numberwithin{equation}{section}
\newcommand{\bQ}{\mathbb{Q}}

\newcommand{\suchthat}{\,:\,}
\newcommand{\spam}{\operatorname{span}}

\newcommand{\Sym}{\ensuremath{\operatorname{Sym}}}

\newcommand{\NSym}{\ensuremath{\operatorname{NSym}}}
\newcommand{\QSym}{\ensuremath{\operatorname{QSym}}}
\newcommand{\dI}{\mathfrak{S}^*} %dual immaculate fns
\newcommand{\rdI}{\mathcal{R}\mathfrak{S}^*} %row-strict dual immaculate fns
\newcommand{\ex}{\mathcal{E}} %extended Schur fns
\newcommand{\rex}{\mathcal{RE}} %row-strict extended Schur fns
\newcommand{\Sk}{\mathscr{S}} %row-strict extended Schur fns
\newcommand{\TdI}{\mathcal{T}_{\alpha/\beta}({\text{1st col}<, \text{rows}\leq})}
\newcommand{\TrdI}{\mathcal{T}_{\alpha/\beta}({\text{1st col}\leq, \text{rows}<})}
\newcommand{\Tex}{\mathcal{T}_{\alpha/\lambda}({\text{cols}<, \text{rows}\leq})}
\newcommand{\Tssy}{\mathcal{T}_{\lambda/\mu}({\text{cols}<, \text{rows}\leq})}
\newcommand{\Trex}{\mathcal{T}_{\alpha/\lambda}({\text{cols}\leq, \text{rows}<})}
\newcommand{\TadI}{\mathcal{T}_{\alpha/\beta}({\text{1st col}<, \text{rows}<})}
\newcommand{\TradI}{\mathcal{T}_{\alpha/\beta}({\text{1st col}\leq, \text{rows}\leq})}
\newcommand{\Tsex}{\mathcal{T}_{\alpha/\lambda}({\text{cols}<, \text{rows}<})}
\newcommand{\Twex}{\mathcal{T}_{\alpha/\lambda}({\text{cols}\leq, \text{rows}\leq})}
\newcommand{\xT}{\mathbf{x}^T}
\DeclareMathOperator{\Des}{Des}

\DeclareMathOperator{\comp}{comp}
\DeclareMathOperator{\asc}{asc}
\DeclareMathOperator{\bA}{\overline{\mathcal{A}}} 
\DeclareMathOperator{\cA}{{\mathcal{A}}}

\newcommand{\refines}{\preccurlyeq}
\newcommand{\coarsens}{\succcurlyeq}

\newlength\cellsize 
\savebox2{%
\begin{picture}(15,15)
\put(0,0){\line(1,0){15}}
\put(0,0){\line(0,1){15}}
\put(15,0){\line(0,1){15}}
\put(0,15){\line(1,0){15}}
\end{picture}}
\newcommand\cellify[1]{\def\thearg{#1}\def\nothing{}%
\ifx\thearg\nothing
\vrule width0pt height\cellsize depth0pt\else
\hbox to 0pt{\usebox2\hss}\fi%
\vbox to 15\unitlength{
\vss
\hbox to 15\unitlength{\hss$#1$\hss}
\vss}}
\newcommand\tableau[1]{\vtop{\let\\=\cr
\setlength\baselineskip{-16000pt}
\setlength\lineskiplimit{16000pt}
\setlength\lineskip{0pt}
\halign{&\cellify{##}\cr#1\crcr}}}
\savebox3{%
\begin{picture}(15,15)
\put(0,0){\line(1,0){15}}
\put(0,0){\line(0,1){15}}
\put(15,0){\line(0,1){15}}
\put(0,15){\line(1,0){15}}
\end{picture}}
\newcommand\expath[1]{%
\hbox to 0pt{\usebox3\hss}%
\vbox to 15\unitlength{
\vss
\hbox to 15\unitlength{\hss$#1$\hss}
\vss}}
\newcommand\bas[1]{\omit \vbox to \cellsize{ \vss \hbox to \cellsize{\hss$#1$\hss} \vss}}

\begin{document}

\title[Symmetric quasisymmetric Schur-like functions]{Symmetric quasisymmetric Schur-like functions}

\author{Maria Esipova}
\address{
 \refy{Department of Mathematics and Statistics,
 McGill University,
 Montreal QC H3A 0B9, Canada}}
\email{\refy{maria.esipova@mail.mcgill.ca}}

\author{Jinting Liang}
\address{
 Department of Mathematics,
 University of British Columbia,
 Vancouver BC V6T 1Z2, Canada}
\email{liangj@math.ubc.ca}

\author{Stephanie van Willigenburg}
\address{
 Department of Mathematics,
 University of British Columbia,
 Vancouver BC V6T 1Z2, Canada}
\email{steph@math.ubc.ca}

\thanks{All authors were supported  in part by the Natural Sciences and Engineering Research Council of Canada.}
\begin{abstract} In this paper we classify when (row-strict) dual immaculate functions and (row-strict) extended Schur functions, \ME{as well as} their skew generalizations, are symmetric. We also classify when \ME{their natural variants,} termed advanced functions, are symmetric. In every case our classification recovers classical skew Schur functions.
\end{abstract}
\maketitle
%\tableofcontents
% Introduction
\section{Introduction}\label{sec:intro}  
The area of quasisymmetric Schur-like functions had its genesis in quasisymmetric Schur functions \cite{QS}, which naturally arose in the study of Macdonald polynomials and refined both Schur functions and their properties. These functions were then applied to resolve whether the algebra of \refy{quasisymmetric} functions, $\QSym$, over the algebra of symmetric functions, $\Sym$, had a stable basis \cite{LM2011}. This is just one example of the plethora of results in which quasisymmetric functions resolved long-standing conjectures on symmetric functions, including the celebrated Shuffle Theorem \cite{CM2018}. Schur-like functions now form a highly active research area, expanding beyond quasisymmetric functions \cite{ALvW2022, monical}, while the foundational research in $\QSym$ and their dual $\NSym$ remains central \svw{\cite{AS2019, allenetal, BBSSZ2014, BLvW, BTvW, Campbell2014, D2024, QS, LOPSvWW25, LM2011, LMvW, mason-remmel, NSvWVW2023, NSvWVW2024}.}

One natural question is when  these quasisymmetric Schur-like functions \refz{are} symmetric. For the original quasisymmetric Schur functions this was classified \refy{by Bessenrodt et al.} \cite{BTvW}, while for another type of Schur-like function, known as the dual immaculate  functions, the classification for those functions indexed by a composition $\alpha$ was revealed \refy{by Allen et al.} \cite[Proposition 4.7]{allenetal}. In this \refy{paper,} we complete and expand their classification by classifying when (row-strict) dual immaculate functions, (row-strict) extended Schur functions and their advanced variants are symmetric for all diagrams $\alpha/\beta$. More \refy{precisely,} our paper is structured as follows.

In Section~\ref{sec:background} we review the necessary background. Then in Section~\ref{sec:symdI} we classify when dual immaculate functions are symmetric in Theorem~\ref{the:symdI}, and similarly when row-strict dual immaculate functions are symmetric in Theorem~\ref{the:symrdI}. Section~\ref{sec:symex} continues our classification, this time for extended Schur functions in Theorem~\ref{the:symex} and row-strict extended Schur functions in Theorem~\ref{the:symrex}. We complete our classification in Section~\ref{sec:symad} with strictly and weakly \svw{advanced  functions} in Theorems~\ref{the:symadI} and \ref{the:symradI}, and strictly and weakly \svw{advanced extended functions} in Theorems~\ref{the:symsex} and \ref{the:symwex}. \svw{We discover in each case that \ME{when these functions are symmetric} we recover particular skew Schur functions, in Corollaries~\ref{cor:symdI}, \ref{cor:symrdI},  \ref{cor:symex}, \ref{cor:symrex}, \finv{\ref{cor:symadI},} \ref{cor:symradI},  \ref{cor:symsex}, and \ref{cor:symwex}, respectively.} \refy{Finally,} in Section~\ref{sec:skew} we end by giving a general combinatorial formula for skew functions in $\QSym$ in Proposition~\ref{prop:skews}.
%Background
\section{Background}\label{sec:background}
%SUBsec:comps
\subsection{Compositions and partitions}\label{subsec:comps}
We say that a \emph{composition} $\alpha = (\alpha_1, \ldots, \alpha_{\ell(\alpha)})$ of $n$ is a list of positive integers such that $\sum _{i=1} ^{\ell(\alpha)} \alpha _i = n$. We denote this by $\alpha \vDash n$. \ME{We} call the $\alpha _i$ the \emph{parts} of $\alpha$, 
$|\alpha|$ the \emph{size} of $\alpha$, and $\ell(\alpha)$ the \emph{length} of $\alpha$. If $\alpha _{j+1}= \cdots = \alpha _{j+m} = i$ we often abbreviate this to $i^m$ and denote the empty composition of 0 by $\emptyset$. We say a composition $\alpha$ is a \emph{partition} if $\alpha _1\geq \cdots \geq \alpha _{\ell(\alpha)}$. We denote this by $\alpha\vdash n$. Note that every composition $\alpha$ determines a partition $\lambda(\alpha)$ by reordering the parts of $\alpha$ in weakly decreasing order. We say that a composition $\beta$ is a \emph{refinement} of $\alpha$, denoted by $\beta\refines \alpha$, if the parts of $\alpha$ \svw{in order} can be obtained by adding together adjacent parts of $\beta$ in order.

\begin{example}\label{ex:refine}
If $\alpha = (3,4,4,3,1)$, then $|\alpha|=15$, $\ell(\alpha) = 5$, $\lambda(\alpha)=(4,4,3,3,1)$ and $(3,2,2,4,3,1)\refines \alpha$.
\end{example}

Given a composition $\alpha$, we say that its \emph{diagram}, also denoted by $\alpha$, is the array of left-justified cells with $\alpha_i$ cells in row $i$ from the \ME{\emph{bottom.}} Given compositions $\alpha, \beta$ we say that $\beta \subseteq \alpha$ if $\beta _i\leq \alpha _i$ for all $1\leq i\leq\ell(\beta)\leq \ell(\alpha)$. Moreover, if $\beta \subseteq \alpha$ then we say that the \emph{(skew) diagram} $\alpha / \beta$ is the array of cells in $\alpha$ but not $\beta$ when $\beta$ is placed in the bottom-left corner of $\alpha$. We omit the word skew if the context is clear, let $|\alpha/\beta| = |\alpha|-|\beta|$, and if $\beta = \emptyset$ then we refer to $\alpha/\beta$ as simply $\alpha$. We say $\alpha/\beta$ is \emph{connected} if we cannot partition the rows of $\alpha/\beta$ into two sets $R_1$ and $R_2$ such that the columns spanned by the cells in the rows of $R_1$ and the rows \svw{of} $R_2$ are disjoint.   We say a skew diagram is an \emph{extended} skew diagram if
\begin{enumerate}
\item $\beta \subseteq \alpha$, and
\item if $\beta _i < \alpha _i$ where $1\leq i \leq \ell (\beta)$, then $\beta _i \geq \beta _j$ for all $i \leq j \leq \ell (\beta)$.
\end{enumerate}
 Informally, this latter condition says that if a cell belongs to $\beta$ then every cell of $\alpha$ below it in that column also belongs to $\beta$. \refz{A reduction at the end of this section will result in us only needing to consider $\beta$ being a partition in this paper.}

Given $\alpha/\beta$ we say a \emph{tableau} $T$ of \emph{shape} $\alpha/\beta$ is a filling of the cells of $\alpha/\beta$ with positive integers, often subject to certain conditions. Moreover, $T$ is a \emph{standard} tableau if the integers $1, \ldots , |\alpha/\beta|$ each appear exactly once. If $c_i(T)$ denotes the number of times $i$ appears in $T$, then we call
$$\ME{c(T) = (c_1(T),  c_{2}(T), \ldots)}$$the \emph{content} of $T$, and given commuting variables $x_1, x_2, \ldots$ let
$$\ME{\xT = x_1^{c_1(T)}x_{2}^{c_{2}(T)}\cdots}$$be the monomial corresponding to $T$. \ME{Let $T(i,j)$ refer to the entry of the cell \finv{$(i,j)$} in row $i$ and column $j$.}

\begin{example}\label{ex:diagrams} We have that $(3,4,4,3,1)/(2,1,2)$ on the left is a diagram, while we have that $(3,4,4,3,1)/(2,2, 1)$ in the middle is an extended diagram, and both are connected. However, $(3,4,4,1,1)/(2,1,2)$ on the right is not connected.
$$\begin{ytableau}
\\
&&\\
\none&\none&&\\
\none&&&\\
\none&\none&\\
\end{ytableau}\qquad
\begin{ytableau}
\\
&&\\
\none&&&\\
\none&\none&&\\
\none&\none&\\
\end{ytableau}\qquad
\begin{ytableau}
\ \\
\ \\
\none&\none&&\\
\none&&&\\
\none&\none&\\
\end{ytableau}$$
Meanwhile for $T$ below, $c(T) = (1,3,3,2,1)$, \svw{with $\xT = x_1x_2^3x_3^3x_4^2x_5$.}
$$T= \begin{ytableau}
5\\
2&2&3\\
\none&\none&3&4\\
\none&1&2&4\\
\none&\none&3
\end{ytableau}$$
\end{example}

%SUBsec QSym and Sym
\subsection{Symmetric and quasisymmetric functions}\label{subsec:symqsym}
We now introduce the Hopf algebras that will be central to our classification. The \emph{Hopf algebra of symmetric functions}, $\Sym$, is the graded Hopf algebra
$$\Sym = \Sym ^0 \oplus \Sym ^1 \oplus \cdots \subset \bQ [[x_1, x_2, \ldots ]]$$where {$[[\cdot ]]$ means that the variables commute,} $\Sym ^0 = \spam \{1\}$ and the $n$th graded piece for $n\geq 1$ is spanned by the basis of monomial symmetric functions, defined as follows. Given $\lambda = (\lambda_1,\ldots ,\lambda_{\ell(\lambda)})\vdash n$, the \emph{monomial symmetric function}, $m_\lambda$, is given by
$$m_\lambda = \sum x_{i_1}^{\lambda _1}\cdots x_{i_{\ell(\lambda)}}^{\lambda _{\ell(\lambda)}}$$summed over \ME{distinct $i_j$ and distinct monomials. We let $m_\emptyset = 1$. Then for $n\geq 0$,
$$\Sym ^n = \spam\{ m_\lambda \suchthat \lambda \vdash n\}.$$}

\begin{example}\label{ex:monomials}
\JL{We have that $m_{(2,1)} = x_1^2x_2 + \finv{x_1x_2^2} + \cdots$.}
\end{example}

The functions that will appear in our classifications can be defined as generating functions of semistandard Young tableaux, which we now define.

\begin{definition}\label{def:ssyt} Let $\lambda, \mu$ be partitions such that $\mu \subseteq \lambda$. Then $$\Tssy$$is the set of all tableaux, called \emph{semistandard Young tableaux}, of  {shape} $\lambda/\mu$ such that
\begin{enumerate}
\item the column entries strictly increase from  bottom to top; 
\item the row entries weakly increase  from left to right.
\end{enumerate}
\end{definition}

\begin{example}\label{ex:ssyt} The following tableau belongs to $\mathcal{T}_{(4,4,3,3,1)/(2,2,1)}({\text{cols}<, \text{rows}\leq})$.
$$\begin{ytableau}
4\\
2&4&5\\
\none&3&3\\
\none&\none&2&3\\
\none&\none&1&2
\end{ytableau}$$
\end{example}

Then for partitions $\lambda, \mu$ such that $\mu \subseteq \lambda$ the \emph{skew Schur function} $s _{\lambda/\mu}$ is given by
\begin{equation}\label{eq:schur}
s _{\lambda/\mu} = \sum _{T\in \Tssy} \xT = \sum _{\nu \vdash |\lambda/\mu|} K _{(\lambda/\mu)\nu} m_\nu
\end{equation}where $K _{(\lambda/\mu)\nu}$ is the \emph{Kostka number} and is the number of $T\in \Tssy$ with content $\nu$. The second equality \refz{is well known.} When $\mu = \emptyset$, we call $s_\lambda$ a \emph{Schur function}.

We now define the \emph{Hopf algebra of quasisymmetric functions}, $\QSym$, which is the graded Hopf algebra
$$\QSym = \QSym ^0 \oplus \QSym ^1 \oplus \cdots \subset \bQ [[x_1, x_2, \ldots ]]$$where $\QSym ^0 = \spam \{1\}$ and the $n$th graded piece for $n\geq 1$ is spanned by the basis of monomial quasisymmetric functions, defined as follows. Given $\alpha = (\alpha_1, \ldots ,\alpha_{\ell(\alpha)})\vDash n$, the \emph{monomial quasisymmetric function}, $M_\alpha$, is given by
$$M_\alpha = \sum _{i_1<\cdots<i_{\ell(\alpha)}} x_{i_1}^{\alpha _1}\cdots x_{i_{\ell(\alpha)}}^{\alpha _{\ell(\alpha)}}.$$\ME{We let $M_\emptyset = 1$. Then for $n\geq 0$,
$$\QSym ^n = \spam\{ M_\alpha \suchthat \alpha \vDash n\}.$$}Observe that by definition
$$m_\lambda = \sum _{\alpha \atop \lambda(\alpha)= \lambda} M_\alpha$$and hence if a function $f\in \QSym$ also satisfies $f\in \Sym$ then, letting $[M_\alpha]f$ denote the coefficient of $M_\alpha$ in $f$, it must be that $[M_\alpha]f= [M_{\alpha '}]f$ for all \ME{$\alpha, \alpha '$ such that} $\lambda(\alpha)=\lambda(\alpha ')$. \ME{The converse also holds.} This observation will be used repeatedly to derive our classifications.

\begin{example}\label{ex:qmonomials}
\JL{We have that $m_{(2,1)} = M_{(2,1)}+M_{(1,2)}$.}
\end{example}

Another central tool will be the involution $\psi$ on $\QSym$, defined on the monomial quasisymmetric function $M_\alpha$ by \cite[Section 5]{E96}
$$\psi(M_\alpha) = (-1)^{|\alpha|-\ell(\alpha)}\sum _{\beta \atop \svw{\beta \coarsens \alpha}} M_\beta$$that on a skew Schur function $s_{\lambda/\mu}$ becomes
\begin{equation}\label{eq:transpose}\psi(s_{\lambda/\mu})=s_{(\lambda/\mu)^t}\end{equation}where  the \emph{transpose} of $\lambda/\mu$ \finv{for partitions $\lambda, \mu$,} denoted by $(\lambda/\mu)^t$, is $\lambda/\mu$ with the rows and columns exchanged. \refz{Note from this equation, it follows that $\psi$ reduces to the well-known involution $\omega$ on $\Sym$.}

\begin{example}\label{ex:transpose} We have that $\left( (4,4,3,3,1)/(2,2,1) \right) ^t$ is the following.
$$
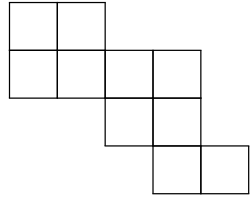
\begin{ytableau}
\ &\ \\
&&&\\
\none&\none& & \\
\none&\none&\none&&\\
\end{ytableau}$$
\end{example}

Finally, we reduce the number of cases we need to consider classifying. Each of our eight quasisymmetric Schur-like functions will be defined as a generating function of specific tableaux, where restrictions are on the columns and on the rows. In particular, given compositions $\alpha, \beta$ with $\beta\subseteq\alpha$, if a diagram $\alpha/\beta$ is disconnected, then we can partition the rows of $\alpha/\beta$ into two sets $R_1, R_2$ such that the columns spanned by the cells in the rows of $R_1$ and the rows \svw{of} $R_2$ are disjoint. \refy{Thus,} the cells in $R_1$ and $R_2$ can be filled independently, and hence we can factor the generating function \refz{ as a product over the connected components of $\alpha/\beta$.} \ME{By the following lemma,} for our \finv{classification} we need only consider \emph{connected} $\alpha/\beta$.

\begin{lemma}[\svw{\cite[Corollary 3.8]{gps}}]\label{lem:symprod}
    Suppose $f=g\cdot h$ for some $f,g,h\in \mathbb{Q}[[x_1,x_2,\ldots]]$. Then $f$ is symmetric if and only if $g$ and $h$ are both symmetric. 
\end{lemma}

Conversely, any product of (skew) Schur functions can be considered as a single skew Schur function of a disconnected skew diagram.

Furthermore, if $\alpha/\beta$ contains a row $r_i$ with no cells, then the generating function it produces is the same as that with the empty row  $r_i$ removed.  \refy{Thus,} we only need \ME{to} consider $\alpha/\beta$ with
$$\beta _i < \alpha _i \mbox{ for all } 1\leq i\leq \ell(\beta).$$In particular, observe that if $\alpha/\beta$ is an extended  skew diagram with $\beta \subseteq \alpha$ then $\beta$ must be a {partition}. 

\

\emph{Thus, for the remainder of the paper, we assume that all diagrams are connected and every row contains at least one cell.}

%SymDI
\section{Symmetric (row-strict) dual immaculate functions}\label{sec:symdI}
For our first classification we focus on dual immaculate functions, before applying $\psi$ to immediately obtain the classification for row-strict dual immaculate functions.

%%%SUB dI
\subsection{Dual immaculate functions}\label{subsec:di} These functions, introduced \refy{by Berg et al.} \cite{BBSSZ2014}, can be defined as generating functions of \emph{immaculate tableaux}, which we now define.

\begin{definition}\label{def:dItab} Let $\alpha, \beta$ be compositions such that $\beta \subseteq \alpha$. Then $$\TdI$$is the set of all tableaux of  {shape} $\alpha/\beta$ such that
\begin{enumerate}
\item the leftmost column entries that belong to $\alpha$ but not $\beta$ strictly increase from  bottom to top; 
\item the row entries weakly increase  from left to right.
\end{enumerate}
\end{definition}

\begin{example}\label{ex:dItab} The following tableau belongs to $\mathcal{T}_{(3,4,4,3,1)/(2,1,2)}({\text{1st col}<, \text{rows}\leq})$. \ME{The entries in the leftmost column are italicised.}
$$\begin{ytableau}
\ME{\emph{5}}\\
\ME{\emph{2}}&2&3\\
\none&\none&3&4\\
\none&1&2&4\\
\none&\none&3
\end{ytableau}$$
\end{example}

Then for compositions $\alpha, \beta$ such that $\beta \subseteq \alpha$ the \emph{dual immaculate function} $\dI _{\alpha/\beta}$ is given by
\begin{equation}\label{eq:dI}
\dI _{\alpha/\beta} = \sum _{T\in \TdI} \xT = \sum _{\gamma \vDash |\alpha/\beta|} c _{(\alpha/\beta)\gamma} M_\gamma
\end{equation}where $c _{(\alpha/\beta)\gamma}$ is the number of $T\in \TdI$ with content $\gamma$. The second equality follows by definition. 

\refy{\begin{example}\label{ex:dI}
We have that $\dI _{(1,2)}= M_{(1,2)}+M_{(1,1,1)}$ from the following tableaux.
$$\begin{ytableau} 2&2\\1\end{ytableau}\qquad \begin{ytableau} 2&3\\1\end{ytableau}$$
Similarly,
$$\dI _{(2,2)/(1)}= M_{(3)}+ 2M_{(2,1)}+2M_{(1,2)}+3M_{(1,1,1)}= m_{(3)}+2m_{(2,1)}+3m_{(1,1,1)}$$from the following tableaux.
$$\begin{ytableau} 1&1\\\none&1\end{ytableau}\qquad 
\begin{ytableau} 1&1\\\none&2\end{ytableau}\qquad 
\begin{ytableau} 1&2\\\none&1\end{ytableau}\qquad 
\begin{ytableau} 1&2\\\none&2\end{ytableau}\qquad 
\begin{ytableau} 2&2\\\none&1\end{ytableau}\qquad 
\begin{ytableau} 1&2\\\none&3\end{ytableau}\qquad 
\begin{ytableau} 1&3\\\none&2\end{ytableau}\qquad 
\begin{ytableau} 2&3\\\none&1\end{ytableau}$$
\end{example}}

We can now classify when dual immaculate functions are symmetric. The case $\beta= \emptyset$ was proved \refy{by Allen et al.} \finv{\cite[Proposition 4.7]{allenetal}} using Young quasisymmetric Schur functions. For completeness, we provide a direct proof of this case, which we then apply to obtain the general case.

\begin{theorem}\label{the:symdI} Let $\alpha, \beta$ be compositions such that $\beta\subseteq \alpha$. Then
$$\dI _{\alpha/\beta} \mbox{ is symmetric} \Leftrightarrow \alpha _i = 1 \mbox{ for } \ell(\beta)+2 \leq i \leq \ell(\alpha).$$
\end{theorem}

\begin{proof} {In this \refy{proof,} we will compute the expansion of \svw{$\dI _{\alpha/\beta}$} in terms of monomial quasisymmetric functions using Equation~\eqref{eq:dI}, and so need to calculate all \ME{$T\in \TdI$} that have content $\gamma$ where $\gamma \vDash |\alpha/\beta|$.} We begin with the case $\beta = \emptyset$. Note that in this case, \ME{for the statement to hold} if $\alpha \vDash n$ then $\alpha =(k,1^{n-k})$ for some $1\leq k\leq n$ is the composition under consideration, \svw{which corresponds to a column if $k=1$ and a row if $k=n$.}

For one direction, if $\alpha=(k,1^{n-k})$, then $\alpha$ is a partition and  
$$\mathcal{T}_{\alpha}({\text{1st col}<, \text{rows}\leq})=\mathcal{T}_{\alpha}({\text{cols}<, \text{rows}\leq}).$$Hence, $\dI_\alpha = s_\alpha$ by definition. 

For the other direction, suppose that $\dI _\alpha$ is symmetric. \JL{Let $\ell=\ell(\alpha)$.}
Consider $\gamma=(1^{\ell-1},n-\ell+1)$. Notice that we can always construct $T\in  \mathcal{T}_{\alpha}({\text{1st col}<, \text{rows}\leq})$ with content $\gamma$ by filling the cells in the leftmost column with $1,\ldots,\ell$, increasing from bottom to top, and the rest of \refz{the cells in $\alpha$ with $\ell$.} This is a valid \refy{filling, and} hence $c_{\alpha\gamma}\neq 0$. Since $\dI _\alpha$ is symmetric, $c_{\alpha\gamma '}\neq 0$ for every $\gamma '$ satisfying $\lambda(\gamma) = \lambda(\gamma ')$. In particular, when $\gamma '=(n-\ell+1,1^{\ell-1})$, $c_{\alpha \gamma '}\neq 0$. Observe that any tableau in $\mathcal{T}_{\alpha}({\text{1st col}<, \text{rows}\leq})$ with content $\gamma '$ must have the cells in the leftmost column filled with $1,\ldots,\ell$ increasing from bottom to top, which forces the remaining cells to be filled with $1$. As all row entries must be weakly increasing from left to right by definition, it follows that only the bottom row can have more than one cell, or equivalently, $\alpha=(k,1^{n-k})$ for some \svw{$1\leq k\leq n$.}
   
    We now proceed with the case \svw{$\beta \neq \emptyset$.} First observe that if we set $D=\alpha/\beta$, then it contains two subdiagrams:  $D_1$ from row $\ell(\beta)+1$ \svw{to} row $\ell(\alpha)$, and the remainder $D_2$, consisting of the bottom $\ell(\beta)$ rows. \JL{So, in particular, $D_1=(\alpha_{\ell(\beta)+1},\ldots,\alpha_{\ell(\alpha)})$.}

    Any tableau $T\in  \TdI$ of shape $D$ restricted to $D_1$ is a tableau \JL{in $\mathcal{T}_{D_1}({\text{1st col}<, \text{rows}\leq})$} and vice versa. \refy{Hence,} the fillings of this subdiagram $D_1$
 contribute $\dI_{D_1}$ to $\dI_{\alpha/\beta}$. This is because the filling of the cells of $D_1$ is independent from that of $D_2$ because $D_2$ has no cells in the leftmost column. Moreover, because $D_2$ has no cells in the leftmost column of $D$, by the definition of $\TdI$, the cells in every row of $D_2$ need to be filled with entries weakly increasing from left to right, and there is no restriction for the entries in the columns. Therefore, the generating function for $D_2$ is 
    \[
    \prod_{i=1}^{\ell(\beta)} s_{(\alpha_i-\beta_i)}.
    \]
    \refy{Hence,} 
    \[\dI_{\alpha/\beta}=\dI_{D_1} \prod_{i=1}^{\ell(\beta)} s_{(\alpha_i-\beta_i)}. 
    \]
    It follows from Lemma~\ref{lem:symprod} that $\dI_{\alpha/\beta}$ is symmetric if and only if $\dI_{D_1}$ is symmetric, and by the first part of this proof, this is the case if and only if $(\alpha_{\ell(\beta)+1},\ldots,\alpha_{\JL{\ell(\alpha)}})=(k,1,\ldots,1)$ \svw{for some $k\geq 1$,} or equivalently, $\alpha _i = 1 \mbox{ for } \ell(\beta)+2 \leq i \leq \ell(\alpha)$.
\end{proof}

{\begin{example}\label{ex:dIproof}
The following demonstrates the construction in the first part of the proof,  \finv{where $\alpha = (3,4,4,3,1).$}
$$\begin{ytableau}
5\\
4&5&5\\
3&5&5&5\\
2&5&5&5\\
1&5&5
\end{ytableau}$$
\end{example}

Thus, by the definition of skew Schur functions, we get the following. 
 \begin{corollary}\label{cor:symdI}
 Let $\alpha, \beta$ be compositions such that $\beta\subseteq \alpha$ and $\dI _{\alpha/\beta}$ is symmetric. Then
 $$\dI _{\alpha/\beta} = s_{(\alpha _{\ell(\beta)+1}, 1^{\ell(\alpha)-\ell(\beta)-1})}\left( \prod _{i=1} ^{\ell(\beta)} s_{(\alpha _i - \beta _i)}\right).$$
 \end{corollary}
 
\begin{example}\label{ex:symdI}
$\dI _{(3,4,4,3,1)/(2,1,2)} = s_{(3,1)}s_{(1)}s_{(3)}s_{(2)}$ is symmetric.
\end{example}

%%%SUB rdI
\subsection{Row-strict dual immaculate functions}\label{subsec:rdI} These \ME{functions} were discovered \refy{by Niese at al.} \cite{NSvWVW2023} and can also be realised as generating functions of tableaux.

\begin{definition}\label{def:rdItab} Let $\alpha, \beta$ be compositions such that $\beta \subseteq \alpha$. Then $$\TrdI$$is the set of all tableaux of {shape} $\alpha/\beta$ such that
\begin{enumerate}
\item the leftmost column entries that belong to $\alpha$ but not $\beta$ weakly increase from  bottom to top; 
\item the row entries strictly increase  from left to right.
\end{enumerate}
\end{definition}

\begin{example}\label{ex:rdItab} The following tableau belongs to $\mathcal{T}_{(3,4,4,3,1)/(2,1,2)}({\text{1st col}\leq, \text{rows}<})$. \ME{The entries in the leftmost column are italicised.}
$$\begin{ytableau}
\ME{\emph{2}}\\
\ME{\emph{2}}&3&5\\
\none&\none&3&4\\
\none&1&2&4\\
\none&\none&3
\end{ytableau}$$
\end{example}

Then for compositions $\alpha, \beta$ such that $\beta \subseteq \alpha$ the \emph{row-strict dual immaculate function} $\rdI _{\alpha/\beta}$ is given by
\begin{equation}\label{eq:rdI}
\rdI _{\alpha/\beta} = \sum _{T\in \TrdI} \xT = \sum _{\gamma \vDash |\alpha/\beta|} rc _{(\alpha/\beta)\gamma} M_\gamma
\end{equation}where $rc _{(\alpha/\beta)\gamma}$ is the number of $T\in \TrdI$ with content $\gamma$. The second equality follows by definition.

\refy{\begin{example}\label{ex:rdI}
We have that $\rdI _{(1,2)}= M_{(2,1)}+M_{(1,1,1)}$ from the following tableaux.
$$\begin{ytableau} 1&2\\1\end{ytableau}\qquad \begin{ytableau} 2&3\\1\end{ytableau}$$
Similarly,
$$\rdI _{(2,2)/(1)}= M_{(2,1)}+M_{(1,2)}+3M_{(1,1,1)}= m_{(2,1)}+3m_{(1,1,1)}$$from the following tableaux. 
$$
\begin{ytableau} 1&2\\\none&1\end{ytableau}\qquad 
\begin{ytableau} 1&2\\\none&2\end{ytableau}\qquad 
\begin{ytableau} 1&2\\\none&3\end{ytableau}\qquad 
\begin{ytableau} 1&3\\\none&2\end{ytableau}\qquad 
\begin{ytableau} 2&3\\\none&1\end{ytableau}
$$
\end{example}}

As proved \refy{by Niese at al.}  \cite[Theorem 3.8]{NSvWVW2023}, for any composition $\alpha$ we have that
$$\psi(\dI _\alpha) = \rdI _\alpha$$and the proof \svw{extends} to diagrams $\alpha/\beta$ for $\beta$ a composition, to give
$$\psi(\dI _{\alpha/\beta}) = \rdI _{\alpha/\beta}.$$\refy{Comparing  Examples~\ref{ex:dI} and \ref{ex:rdI}, note that $\psi$ does not necessarily preserve the number of terms in the monomial quasisymmetric function expansion. Using $\psi$ we} immediately obtain the following two results for \svw{row-strict} dual immaculate functions from Theorem~\ref{the:symdI} and Corollary~\ref{cor:symdI}, recalling \svw{Equation~\eqref{eq:transpose}} for the latter.

\begin{theorem}\label{the:symrdI} Let $\alpha, \beta$ be compositions such that $\beta\subseteq \alpha$. Then
$$\rdI _{\alpha/\beta} \mbox{ is symmetric} \Leftrightarrow \alpha _i = 1 \mbox{ for } \ell(\beta)+2 \leq i \leq \ell(\alpha).$$
\end{theorem}

 \begin{corollary}\label{cor:symrdI}
 Let $\alpha, \beta$ be compositions such that $\beta\subseteq \alpha$ and $\rdI _{\alpha/\beta}$ is symmetric. Then
 $$\rdI _{\alpha/\beta} = s_{\refy{(\ell(\alpha)-\ell(\beta), 1^{\finv{-1+\alpha _{\ell(\beta)+1}}})}}\left( \prod _{i=1} ^{\ell(\beta)} s_{(1^{\alpha _i - \beta _i})}\right).$$
 \end{corollary}
 
 \begin{example}\label{ex:symrdI}
$\rdI _{(3,4,4,3,1)/(2,1,2)} = s_{(2, 1^2)}s_{(1)}s_{(1^3)}s_{(1^2)}$ is symmetric.
\end{example}
%%%Sym ex
\section{Symmetric (row-strict) extended Schur functions}\label{sec:symex} The extended Schur functions are natural generalizations of Schur functions and, as we will see, extend the definition of Schur functions in Section~\ref{sec:background} by elegantly generalizing the definition of semistandard Young tableaux from partition shape to composition shape. These functions arose naturally in two areas: as dual functions to shin functions \cite{Campbell2014} and as stable limits of lock polynomials \cite{AS2019}.

%%%SUB extended Schurs
\subsection{Extended Schur functions}\label{subsec:ex}
These functions are generating functions of \emph{extended tableaux} that coincide with semistandard Young tableaux when we restrict our diagrams to partitions. Note that in the next definition the partition $\lambda$ can be replaced by a composition $\beta$ \cite[Theorem 17]{LOPSvWW25}, however, we will not need this here. 

\begin{definition}\label{def:extab} Let $\alpha$ be a composition, and $\lambda$ be a partition such that $\lambda \subseteq \alpha$. Then $$\Tex$$is the set of all tableaux of {shape} $\alpha/\lambda$ such that
\begin{enumerate}
\item the column entries strictly increase from  bottom to top; 
\item the row entries weakly increase  from left to right.
\end{enumerate}
\end{definition}

\refy{These tableaux are also known as shin-tableaux when $\lambda = \emptyset$ \cite{Campbell2014}.}

\begin{example}\label{ex:extab} The following tableau belongs to $\mathcal{T}_{(3,4,4,3,1)/(2,2,1)}({\text{cols}<, \text{rows}\leq})$.
$$\begin{ytableau}
3\\
2&3&5\\
\none&2&4&4\\
\none&\none&2&3\\
\none&\none&1
\end{ytableau}$$
\end{example}

\refy{The following definition uses the notation of Assaf and Searles \cite{AS2019}, however, if we instead used the notation of dual shin functions, then $\ex _{\alpha}$ would be denoted by $\shin^*_{\alpha}$. For} a composition $\alpha$ and partition $\lambda$ such that $\lambda \subseteq \alpha$ the \emph{extended Schur function} $\ex _{\alpha/\lambda}$ is given by
\begin{equation}\label{eq:ex}
\ex _{\alpha/\lambda} = \sum _{T\in \Tex} \xT = \sum _{\gamma \vDash |\alpha/\lambda|} d _{(\alpha/\lambda)\gamma} M_\gamma
\end{equation}where $d _{{(\alpha/\lambda)}\gamma}$ is the number of $T\in \Tex$ with content $\gamma$. The second equality follows by definition. 

\refy{\begin{example}\label{ex:ex}
We have that $\ex _{(1,2)}= M_{(1,2)}+M_{(1,1,1)}$ from the following tableaux.
$$\begin{ytableau} 2&2\\1\end{ytableau}\qquad \begin{ytableau} 2&3\\1\end{ytableau}$$
Similarly,
$$\ex _{(2,2)/(1)}= M_{(2,1)}+M_{(1,2)}+2M_{(1,1,1)}= m_{(2,1)}+2m_{(1,1,1)}$$from the following tableaux.
$$
\begin{ytableau} 1&2\\\none&1\end{ytableau}\qquad 
\begin{ytableau} 2&2\\\none&1\end{ytableau}\qquad 
\begin{ytableau} 1&3\\\none&2\end{ytableau}\qquad 
\begin{ytableau} 2&3\\\none&1\end{ytableau}$$
\end{example}}

The necessity for \ME{$\lambda$ being a partition} is discussed in the final section, for the interested reader.
For now, we give our classification for when an extended Schur function is symmetric.

\begin{theorem}\label{the:symex}
Let $\alpha$ be a composition and $\lambda$ be a partition, such that $\lambda \subseteq \alpha$. Then
$$\ex _{\alpha/\lambda} \mbox{ is symmetric} \Leftrightarrow \alpha  \mbox{ is a partition. }$$
\end{theorem}

\begin{proof} 
For one direction, observe that if both $\alpha$ and $\lambda$ are partitions, then $\ex_{\alpha/\lambda}=s_{\alpha/\lambda}$ \refz{by comparing their definitions, and hence  $\ex_{\alpha/\lambda}$ is symmetric.} For the other direction, suppose to the contrary that $\alpha/\lambda$ \refz{is} such that $\lambda\subseteq \alpha$, $\lambda$ is a partition, and $\alpha$ is a composition that is not a partition. Denote by $S_{\alpha/\lambda}^{\gamma}$ the set of tableaux in $\Tex$ with content $\gamma$.
    
    In order to show that $\ex_{\alpha/\lambda}$ is not symmetric, by the definition of $\ex_{\alpha/\lambda}$ in terms of monomial quasisymmetric functions \svw{in Equation~\eqref{eq:ex},} it suffices to find some $\gamma$ and $\gamma '$, such that $d _{(\alpha/\lambda)\gamma} \neq d _{(\alpha/\lambda)\gamma'}$ but $\lambda(\gamma)=\lambda(\gamma ')$.

    Let us start by constructing $\gamma$ and $\gamma'$. Take the maximal $I$ such that $\alpha_{I+1}>\alpha_I$, \ME{which exists because $\alpha$ is not a partition.} Let 
    \[
K=\sum_{i=1}^I(\alpha_i-\lambda_i)+(\alpha_I-\lambda_{I+1})-1.
    \]
    Then set 
    \[
   \gamma=(1^{K-1},2,1^{\refz{|\alpha/\lambda|}-K-1}) \quad \text{ and }\quad  \gamma'=(1^{K},2,1^{\refz{|\alpha/\lambda|}-K-2}).
    \]
  Note that $\lambda(\gamma)=\lambda(\gamma ')=(2, 1^{\refz{|\alpha/\lambda|}-2})$. 
    
    Next we will construct an injective yet not surjective map $$\phi:S_{\alpha/\lambda}^{\gamma}\to S_{\alpha/\lambda}^{\gamma'}$$from which it follows that $d _{(\alpha/\lambda)\gamma}<d _{(\alpha/\lambda)\gamma'}$.
     Given $T \in \Tex$, each column contains either {both} $K, K+1$, exactly one of $K, K+1$, or neither. If both $K, K+1$ are contained in a column then we call them a \emph{pair}. We call the pairs \emph{fixed} and all other occurrences of $K$ and $K+1$ \emph{free}. For a given $T\in S_{\alpha/\lambda}^{\gamma}$, construct $\phi(T)$ from $T$ by doing the following swapping procedure for each row, \refz{which is similar to the Bender-Knuth involution.} If a row has $p$ free $K$'s followed by $q$ free $(K+1)$'s, then replace them by $q$ free $K$'s followed by $p$ free $(K+1)$'s. 
     
   Observe that \svw{the} resulting tableau $\phi(T)$ has content $\gamma'$. We still need to verify that $\phi(T)\in \Tex$. Note that \ME{since $T$ has content $\gamma$} there are only two copies of $K$ and one copy of $K+1$ in $T$, and swapping $K$'s and $K+1$'s will not affect the increasing row and column criteria with regard to entries not equal to $K$ or $K+1$. So it suffices to see if the restrictions still hold among the three cells whose entries have been swapped. 
     \begin{enumerate}
         \item If all of them are free, then $\phi(T)\in \Tex$ by the above note.
         \item If a pair of $K$ and $K+1$ is fixed and the other $K$ is free, depending on  the location of the free $K$, we have the following cases.
     \begin{enumerate}
         \item  If the free $K$ is in a cell sharing no common row with the pair, then $\phi(T)\in \Tex$ by the above note.
         \item If the free $K$ is in the cell to the immediate left of the $K+1$, or to the immediate right of the $K$ of the pair, then all rows will still be weakly increasing so $\phi(T)\in \Tex$ by the above note.
         \item If the free $K$ is in the cell to the immediate left of the $K$ of the pair, then assume \ME{$T(i,h)=T(i,h+1)=K$ and $T(j,h+1)=K+1$ for some row $j>i$, and column $h$. Observe that $(j,h)\in\alpha/\lambda$ because $\lambda$ is a partition. By the definition of $\Tex$, $T(j,h)=K+1$,} which makes it impossible for $T$ to have content $\gamma$, therefore this case does not occur.  
     \end{enumerate}
     \end{enumerate}
This verifies that $\phi$ is well-defined. 
     
     Observe that $\phi$ is an injection, so we only need to show that there exists some $T'\in S_{\alpha/\lambda}^{\gamma'}$ with no preimage under the map $\phi$. We construct $T'$ as follows.
     \begin{enumerate}
         \item Fill the first $I-1$ rows from bottom to top, from left to right with \ME{$1,2,\ldots,\sum_{i=1}^{I-1}(\alpha_{i}-\lambda_{i})$} in order.
         \item For row $I$, fill all but the last cell with $\sum_{i=1}^{I-1}(\alpha_{i}-\lambda_{i})+1,\ldots,\sum_{i=1}^{I-1}(\alpha_{i}-\lambda_{i})+(\alpha_I-\lambda_I)-1$, and put $K$ in the last cell. 
         \item For the next row $I+1$, fill the first $\alpha_I-\lambda_{I+1}-1$ cells with $\sum_{i=1}^{I}(\alpha_{i}-\lambda_{i}),\ldots, K-1$, followed by two cells with $(K+1)$'s. 
         \item Continue filling the rest from bottom to top, from left to right with $K+2,\ldots,\svw{|\alpha/\lambda|-1}$ \ME{in order.}
     \end{enumerate}
         \refz{By the} construction \refz{of $T'$ and definition of $K$, we know that almost all of} the rows/columns are weakly/strictly increasing going left to right/bottom to top, respectively, with the possible exception  \refz{of the cells in row $I$ and $I+1$ that lie in column $\alpha_I$, for which this property needs to be checked.} However,  $T'(I,\alpha_I)=K$, $T'(I+1,\alpha_I)=K+1$ and $T'(I+1,\alpha_I+1)=K+1$, so column \refz{$\alpha_I$} increases from bottom to top. \refz{Hence,   $T'\in S_{\alpha/\lambda}^{\gamma'}$.} Moreover,  the first two \finv{entries} are fixed and the last $K+1$ is free. Therefore, $T'$ has no preimage under $\phi$, and we are done.
\end{proof}

\begin{example}\label{ex:Tprime}
We construct $T'$ for $\alpha = (6,4,6,2)$ and $\lambda = (2,1,1)$. Then $I=2$ and $K=4+3+(4-1) -1=9$ and 
$$T'= \begin{ytableau}
12&13\\
\none&7&8&10&10&11\\
\none&5&6&9\\
\none&\none&1&2&3&4\\
\end{ytableau}\quad.$$
\end{example}

Hence, by the definition of skew Schur functions, we get the following.

\begin{corollary}\label{cor:symex}
Let $\alpha$ be a composition and $\lambda$ be a partition, such that $\lambda \subseteq \alpha$ and $\ex _{\alpha/\lambda}$ is symmetric. Then
 $$\ex _{\alpha/\lambda} = s_{\alpha/\lambda}.$$
\end{corollary}

\begin{example}\label{ex:symex} $\ex _{(3,4,4,3,1)/(2,2,1)}$ is not symmetric, however $\ex _{(4,4,3,3,1)/(2,2,1)} = s_{(4,4,3,3,1)/(2,2,1)}$ is symmetric.
\end{example}

%%%SUB row-strict extended Schurs
\subsection{Row-strict extended Schur functions}\label{subsec:rex}
As with the row-strict analogue of dual immaculate functions, the row-strict extended Schur functions are related to extended Schur functions via the involution $\psi$.

\begin{definition}\label{def:rextab} Let $\alpha$ be a composition, and $\lambda$ be a partition such that $\lambda \subseteq \alpha$. Then $$\Trex$$is the set of all tableaux of {shape} $\alpha/\lambda$ such that
\begin{enumerate}
\item the column entries weakly increase from  bottom to top; 
\item the row entries strictly increase  from left to right.
\end{enumerate}
\end{definition}

\begin{example}\label{ex:rextab} The following tableau belongs to $\mathcal{T}_{(3,4,4,3,1)/(2,2,1)}({\text{cols}\leq, \text{rows}<})$.
$$\begin{ytableau}
3\\
1&4&5\\
\none&2&3&4\\
\none&\none&2&3\\
\none&\none&2
\end{ytableau}$$
\end{example}

Then for a composition $\alpha$ and partition $\lambda$ such that $\lambda \subseteq \alpha$ the \emph{row-strict extended Schur function} $\rex _{\alpha/\lambda}$ is given by
\begin{equation}\label{eq:rex}
\rex _{\alpha/\lambda} = \sum _{T\in \Trex} \xT = \sum _{\gamma \vDash |\alpha /\lambda|} rd _{(\alpha/\lambda)\gamma} M_\gamma
\end{equation}where $rd _{{(\alpha/\lambda)}\gamma}$ is the number of $T\in \Trex$ with content $\gamma$. The second equality follows by definition.

\refy{\begin{example}\label{ex:rex}
We have that $\rex _{(1,2)}= M_{(2,1)}+M_{(1,1,1)}$ from the following tableaux.
$$\begin{ytableau} 1&2\\1\end{ytableau}\qquad \begin{ytableau} 2&3\\1\end{ytableau}$$
Similarly,
$$\rex _{(2,2)/(1)}= M_{(2,1)}+M_{(1,2)}+2M_{(1,1,1)}= m_{(2,1)}+2m_{(1,1,1)}$$from the following tableaux. 
$$
\begin{ytableau} 1&2\\\none&1\end{ytableau}\qquad 
\begin{ytableau} 1&2\\\none&2\end{ytableau}\qquad 
\begin{ytableau} 1&3\\\none&2\end{ytableau}\qquad 
\begin{ytableau} 2&3\\\none&1\end{ytableau}
$$
\end{example}}

As proved and illustrated \refy{by Niese et al.} \cite[Figure 1]{NSvWVW2024}, for any composition $\alpha$ we have that
$$\psi(\ex _\alpha) = \rex _\alpha$$and the proof extends to diagrams $\alpha/\lambda$, for $\lambda$ a partition, to give
$$\psi(\ex _{\alpha/\lambda}) = \rex _{\alpha/\lambda}.$$Hence, by Theorem~\ref{the:symex},  \svw{Corollary~\ref{cor:symex} and Equation~\eqref{eq:transpose},} we get the following, respectively.

\begin{theorem}\label{the:symrex}
Let $\alpha$ be a composition and $\lambda$ be a partition, such that $\lambda \subseteq \alpha$. Then
$$\rex _{\alpha/\lambda} \mbox{ is symmetric} \Leftrightarrow \alpha  \mbox{ is a partition. }$$
\end{theorem}

\begin{corollary}\label{cor:symrex}
Let $\alpha$ be a composition and $\lambda$ be a partition, such that $\lambda \subseteq \alpha$ and $\rex _{\alpha/\lambda}$ is symmetric. Then
 $$\rex _{\alpha/\lambda} = s_{(\alpha/\lambda)^t}.$$
\end{corollary}

\begin{example}\label{ex:symrex} $\rex _{(3,4,4,3,1)/(2,2,1)}$ is not symmetric, however $\rex _{(4,4,3,3,1)/(2,2,1)} = s_{(5,4,4,2)/(3,2)}$ is symmetric.
\end{example}

%%%Sec: Advance!
\section{Symmetric advanced functions}\label{sec:symad}
We now classify symmetry for functions that are variants of the (row-strict) dual immaculate functions and (row-strict) extended Schur functions. Unlike these aforementioned functions, these variants are not bases for $\QSym$ because in each case the functions indexed by $(n)$ and $(1^n)$ are equal. \refz{These functions} were originally introduced \refy{by Niese et al.} \cite{NSvWVW2024}  with respect to descent sets, which \svw{they} described in the final section, and denoted by \ME{the letter} $\mathcal{A}$.  Therefore, we term these \svw{advanced functions} or \svw{advanced extended} functions, with the descriptor of strictly or weakly, depending on the row and column conditions.

%%%SUBSec: Strong duals
\subsection{Strictly \svw{advanced functions}}\label{subsec:sadi} We begin with strict inequalities.
\begin{definition}\label{def:sadItab} Let $\alpha, \beta$ be compositions such that $\beta \subseteq \alpha$. Then $$\TadI$$is the set of all tableaux of {shape} $\alpha/\beta$ such that
\begin{enumerate}
\item the leftmost column entries that belong to $\alpha$ but not $\beta$ strictly increase from  bottom to top; 
\item the row entries strictly increase  from left to right.
\end{enumerate}
\end{definition}

\begin{example}\label{ex:sdItab} The following tableau belongs to $\mathcal{T}_{(3,4,4,3,1)/(2,1,2)}({\text{1st col}<, \text{rows}<})$. \ME{The entries in the leftmost column are italicised.}
$$\begin{ytableau}
\ME{\emph{3}}\\
\ME{\emph{2}}&3&5\\
\none&\none&3&4\\
\none&1&2&4\\
\none&\none&3
\end{ytableau}$$
\end{example}

Then for compositions $\alpha, \beta$ such that $\beta \subseteq \alpha$ the \emph{strictly \svw{advanced function}} $\bA\dI _{\alpha/\beta}$ is given by
\begin{equation}\label{eq:sadI}
\bA\dI _{\alpha/\beta} = \sum _{T\in \TadI} \xT = \sum _{\gamma \vDash |\alpha/\beta|} s{c} _{(\alpha/\beta)\gamma} M_\gamma
\end{equation}where $s{c} _{(\alpha/\beta)\gamma}$ is the number of $T\in \TadI$ with content $\gamma$. The second equality follows by definition.

\begin{theorem}\label{the:symadI} Let $\alpha, \beta$ be compositions such that $\beta\subseteq \alpha$. Then
$$\bA\dI _{\alpha/\beta} \mbox{ is symmetric} \Leftrightarrow \alpha _i = 1 \mbox{ for } \ell(\beta)+1 \leq i \leq \ell(\alpha)-1.$$
\end{theorem}

\begin{proof} {In this \refy{proof,} we will compute the expansion of $\bA\dI _{\alpha/\beta}$ in terms of monomial quasisymmetric functions using Equation~\eqref{eq:sadI}, and so need to calculate all $T\in \TadI$ that have content $\gamma$ where $\gamma \vDash |\alpha/\beta|$.}

We begin with the case $\beta = \emptyset$. Note that in this case, \ME{for the statement to hold} if $\alpha \vDash n$ then $\alpha=(1^{k},n-k)$ for some \finv{$0\leq k\leq n-1$} is the composition under consideration, \svw{which corresponds to a row if $k=0$ and a column if \finv{$k=n-1$.}} For one direction, let $\alpha=(1^{k},n-k)$.  Then for every $T\in \mathcal{T}_{\alpha}({\text{1st col}<, \text{rows}<})$ the entries in the leftmost column must  strictly increase and be $1,2,\ldots,k+1$. Thereafter, the entries in the row of length $n-k$  must be $k+1,\ldots,n$ from left to right, by definition. This tableau $T$ is unique, and therefore we get $$\JL{\bA\dI _{\alpha}=M_{(1^n)}=s_{(1^n)}}$$by definition, which is symmetric.

For the other direction,  if $\alpha \neq (1^{k},n-k)$, then in the leftmost column of $\alpha$ there must be at least one row with more than one cell, and at least one cell above it. Choose the lowest such row, row $i$. Then  fill the cells of the leftmost column with $1,\ldots,i$ going up the column. Then fill the cell to the immediate right and above the cell filled with $i$ each with an $i+1$; fill the remaining cells with $i+2,\ldots,n-1$ \ME{in order from the bottom row, left to right.} So the filling of this tableau $T\in \mathcal{T}_{\alpha}({\text{1st col}<, \text{rows}<})$ contains two $i+1$'s and only one $1$. Now observe that there is no way to fill a tableau of shape $\alpha$ with two 1's such that it belongs to $\mathcal{T}_{\alpha}({\text{1st col}<, \text{rows}<})$. So our function indexed by $\alpha$ is not symmetric. This concludes the case $\beta = \emptyset$.

We now proceed with the case $\beta \neq \JL{\emptyset}$, which is similar to the analogous part of the proof of Theorem~\ref{the:symdI}. If we set $D=\alpha/\beta$, then it contains two subdiagrams:  $D_1$ from row $\ell(\beta)+1$ \svw{to} row $\ell(\alpha)$, and the remainder $D_2$, consisting of the bottom $\ell(\beta)$ rows. \JL{So, in particular, $D_1=(\alpha_{\ell(\beta)+1},\ldots,\alpha_{\ell(\alpha)})$.}

    Any tableau \JL{$T\in  \TadI$} of shape $D$ restricted to $D_1$ is a tableau \JL{in $\mathcal{T}_{D_1}({\text{1st col}<, \text{rows}<})$} and vice versa. 
   \refy{Hence,} the fillings of this subdiagram $D_1$
 contribute $\bA\dI_{D_1}$ to $\bA\dI_{\alpha/\beta}$. This is because the filling of the cells of $D_1$ is independent from that of $D_2$ because $D_2$ has no cells in the leftmost column. Moreover, because $D_2$ has no cells in the leftmost column of $D$, by the definition of $\TadI$, the cells in every row of $D_2$ need to be filled with entries strictly increasing from left to right, and there is no restriction for the entries in the columns. Therefore, the generating function for $D_2$ is 
    \[
    \prod_{i=1}^{\ell(\beta)} s_{(1^{\alpha_i-\beta_i})}.
    \]
    \refy{Hence,} 
    \[\bA\dI_{\alpha/\beta}=\bA\dI_{D_1} \prod_{i=1}^{\ell(\beta)} s_{(1^{\alpha_i-\beta_i})}. 
    \]
    It follows from Lemma~\ref{lem:symprod} that $\bA\dI_{\alpha/\beta}$ is symmetric if and only if $\bA\dI_{D_1}$ is symmetric, and by the first part of this proof, this is the case if and only if $(\alpha_{\ell(\beta)+1},\ldots,\alpha_{\JL{\ell(\alpha)}})=(1,\ldots,1, k)$ \svw{for some $k\geq 1$,} or equivalently, $\alpha _i = 1 \mbox{ for } \ell(\beta)+1 \leq i \leq \ell(\alpha)-1$.
\end{proof}

{\begin{example}\label{ex:Aproof}
The following demonstrates the construction in the first part of the proof, \finv{where $\alpha = (1,3,4,4,3)$ and} $i=2$.
$$\begin{ytableau}
12&13&14\\
8&9&10&11\\
3&5&6&7\\
2&3&4\\
1
\end{ytableau}$$
\end{example}

Thus, by the definition of skew Schur functions, we get the following. 
 \begin{corollary}\label{cor:symadI}
 Let $\alpha, \beta$ be compositions such that $\beta\subseteq \alpha$ and $\bA\dI _{\alpha/\beta}$ is symmetric. Then
 $$\bA\dI _{\alpha/\beta} = s_{(1^{\ell(\alpha)-\ell(\beta)-1+\alpha_{\ell(\alpha)}})}\left( \prod _{i=1} ^{\ell(\beta)} s_{(1^{\alpha _i - \beta _i})}\right).$$
 \end{corollary}
 
\begin{example}\label{ex:symadI}
$\bA\dI _{(3,4,4,1,3)/(2,1,2)} = s_{(1^4)}s_{(1)}s_{(1^3)}s_{(1^2)}$ is symmetric.
\end{example}

%%%SUB weak duals
\subsection{Weakly \svw{advanced functions}}\label{subsec:radI} We now exchange our strict inequalities for weak inequalities.
\begin{definition}\label{def:radItab} Let $\alpha, \beta$ be compositions such that $\beta \subseteq \alpha$. Then $$\TradI$$is the set of all tableaux of {shape} $\alpha/\beta$ such that
\begin{enumerate}
\item the leftmost column entries that belong to $\alpha$ but not $\beta$ weakly increase from  bottom to top; 
\item the row entries weakly increase  from left to right.
\end{enumerate}
\end{definition}

\begin{example}\label{ex:radItab} The following tableau belongs to $\mathcal{T}_{(3,4,4,3,1)/(2,1,2)}({\text{1st col}\leq, \text{rows}\leq})$. \ME{The entries in the leftmost column are italicised.}
$$\begin{ytableau}
\ME{\emph{2}}\\
\ME{\emph{2}}&3&5\\
\none&\none&3&3\\
\none&1&2&4\\
\none&\none&4
\end{ytableau}$$
\end{example}

Then for compositions $\alpha, \beta$ such that $\beta \subseteq \alpha$ the \emph{weakly \svw{advanced  function}} $\cA\dI _{\alpha/\beta}$ is given by
\begin{equation}\label{eq:radI}
\cA\dI _{\alpha/\beta} = \sum _{T\in \TradI} \xT = \sum _{\gamma \vDash |\alpha/\beta|}  wc _{(\alpha/\beta)\gamma} M_\gamma
\end{equation}where $wc _{(\alpha/\beta)\gamma}$ is the number of $T\in \TradI$ with content $\gamma$. The second equality follows by definition.

As shown \refy{by Niese et al.} \cite[Figure 1]{NSvWVW2024}, for any composition $\alpha$ we have that
$$\psi(\bA\dI _\alpha) = \cA\dI _\alpha$$and the proof \svw{extends} to diagrams $\alpha/\beta$ for $\beta$ a composition, to give
$$\psi(\bA\dI _{\alpha/\beta}) = \cA\dI _{\alpha/\beta}.$$Hence, we immediately obtain the following two results for weakly \svw{advanced functions} from Theorem~\ref{the:symadI} and Corollary~\ref{cor:symadI}, using \svw{Equation~\eqref{eq:transpose}} for the latter.

\begin{theorem}\label{the:symradI} Let $\alpha, \beta$ be compositions such that $\beta\subseteq \alpha$. Then
$$\cA\dI _{\alpha/\beta} \mbox{ is symmetric} \Leftrightarrow \alpha _i = 1 \mbox{ for } \ell(\beta)+1 \leq i \leq \ell(\alpha) -1.$$
\end{theorem}

 \begin{corollary}\label{cor:symradI}
Let $\alpha, \beta$ be compositions such that $\beta\subseteq \alpha$ and $\cA\dI _{\alpha/\beta}$ is symmetric. Then
 $$\cA\dI _{\alpha/\beta} = s_{(\ell(\alpha)-\ell(\beta) -1+\alpha_{\ell(\alpha)})}\left( \prod _{i=1} ^{\ell(\beta)} s_{({\alpha _i - \beta _i})}\right).$$
 \end{corollary}
 
\begin{example}\label{ex:symradI}
$\cA\dI _{(3,4,4,1,3)/(2,1,2)} = s_{(4)}s_{(1)}s_{(3)}s_{(2)}$ is symmetric.
\end{example}

%%%SUBSec: Strong extendeds
\subsection{Strictly \svw{advanced extended} functions}\label{subsec:sex}
We now \refy{investigate} our final pair of functions, for which we will use a different approach. 

\begin{definition}\label{def:sextab} Let $\alpha$ be a composition, and $\lambda$ be a partition such that $\lambda \subseteq \alpha$. Then $$\Tsex$$is the set of all tableaux of {shape} $\alpha/\lambda$ such that
\begin{enumerate}
\item the column entries strictly increase from  bottom to top; 
\item the row entries strictly increase  from left to right.
\end{enumerate}
\end{definition}

\begin{example}\label{ex:sextab} The following tableau belongs to $\mathcal{T}_{(3,4,4,3,1)/(2,2,1)}({\text{cols}<, \text{rows}<})$.
$$\begin{ytableau}
3\\
2&4&5\\
\none&2&3&4\\
\none&\none&2&3\\
\none&\none&1
\end{ytableau}$$
\end{example}

Then for a composition $\alpha$ and partition $\lambda$ such that $\lambda \subseteq \alpha$ the \emph{strictly \svw{advanced extended} function} $\bA\ex _{\alpha/\lambda}$ is given by
\begin{equation}\label{eq:sex}
\bA\ex _{\alpha/\lambda} = \sum _{T\in \Tsex} \xT = \sum _{\gamma \vDash |\alpha/\lambda|} sd _{(\alpha/\lambda)\gamma} M_\gamma
\end{equation}where $sd _{{(\alpha/\lambda)}\gamma}$ is the number of $T\in \Tsex$ with content $\gamma$. The second equality follows by definition. The necessity for \ME{$\lambda$ being a partition} is discussed in the final section, for the interested reader.

We now introduce some notation that will allow us to prove our classification succinctly, and it should be noted that we could also prove our result directly by translating the results of \refy{Gillespie et al.} \cite{gps} \finv{wholly} into the language of tableaux.

Following  \refy{Gillespie et al.} \cite{gps}, for every vertex-labelled directed graph $G=(V,E)$ \svw{with vertex labels $1, \ldots, |V|$,} denote by $K(G)$ the set of all \ME{its \svw{\emph{proper colourings}} $\kappa : |V| \rightarrow \{1,2,\ldots\}$ such that if two vertices $v_i, v_j \in V$ are adjacent, then $\kappa(v_i)\neq \kappa(v_j).$} An \emph{ascent} of $\kappa\in K(G)$ is an adjacent pair of labelled vertices $i<j$ with $\kappa(i)<\kappa(j)$. We write $\asc(\kappa)$ for the number of ascents of $\kappa$. Denote by $K(G)^{|E|}$ the set of proper colourings of $G$ with $|E|$ \svw{ascents.}
The {\em chromatic quasisymmetric function}  $X_G(x;q)$ is defined as
\begin{equation*}  X_G(x;q)=\sum_{\kappa\in K(G)}q^{\asc(\kappa)}x_{\kappa(1)}\cdots \svw{x_{\kappa(|V|)}.}
\end{equation*}

We will now construct a vertex-labelled acyclic directed  graph $G_{\alpha/\lambda} = (V,E)$ from the diagram $\alpha/\lambda$ as follows. In a diagram, we say two cells  are \emph{adjacent} in a row (respectively, column) if both cells are in the same row (respectively, column) and there is no cell between them. \refz{Note that if two adjacent cells are in the same column, then there may be empty space between them, as in Example~\ref{ex:Graph}.} Let $V$ be in natural bijection with the cells of the diagram $\alpha/\lambda$, and obtain $E$ by drawing a directed edge between the vertices corresponding to two adjacent cells going from left to right in a row or from bottom to top in a column. The resulting directed graph $G$ is an acyclic directed graph. Finally, label the cells in $\alpha/\lambda$ from left to right in a row, taking the rows from bottom to top. This induces a natural vertex-labelling  on $G$ giving us our vertex-labelled acyclic directed  graph $G_{\alpha/\lambda}$.

\begin{example}\label{ex:Graph} The labelled diagram $(3,4,3,4,1)/(2,2, 1)$ has $G_{(3,4,3,4,1)/(2,2,1)}$ below.
$$
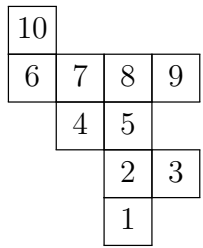
\begin{ytableau}
10 \\
6&7&8&9\\
\none&4&5\\
\none&\none&2 &3 \\
\none&\none&1
\end{ytableau}
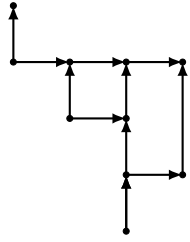
$$$$
%[0.28\textheight]{0.6\textwidth}
\begin{tikzpicture}[scale=.75, baseline]%, transform shape]
\coordinate (A) at (0,4);
\coordinate (B) at (0,3);
\coordinate (C) at (1,3);
\coordinate (D) at (2,3);
\coordinate (E) at (3,3);
\coordinate (F) at (1,2);
\coordinate (G) at (2,2);
\coordinate (H) at (2,1);
\coordinate (I) at (3,1);
\coordinate (J) at (2,0);
\filldraw (A) circle (1.5pt);
\filldraw (B) circle (1.5pt);
\filldraw (C) circle (1.5pt);
\filldraw (D) circle (1.5pt);
\filldraw (E) circle (1.5pt);
\filldraw (F) circle (1.5pt);
\filldraw (G) circle (1.5pt);
\filldraw (H) circle (1.5pt);
\filldraw (I) circle (1.5pt);
\filldraw (J) circle (1.5pt);
\draw[thick, -latex] (B) -- (A);
\draw[thick, -latex] (B) -- (C);
\draw[thick, -latex] (C) -- (D);
\draw[thick, -latex] (D) -- (E);
\draw[thick, -latex] (F) -- (C);
\draw[thick, -latex] (F) -- (G);
\draw[thick, -latex] (G) -- (D);
\draw[thick, -latex] (H) -- (G);
\draw[thick, -latex] (J) -- (H);
\draw[thick, -latex] (H) -- (I);
\draw[thick, -latex] (J) -- (H);
\draw[thick, -latex] (I) -- (E);
\end{tikzpicture}$$
%	\coordinate (A) at (0,0);
%\coordinate (B) at (1,0);
%\coordinate (C) at (0.5,.75);
%	\draw[thick, ->] (A) -- (C);
%	\draw[thick, ->] (A) -- (B);
%\draw[thick, ->] (B) -- (C);
%\filldraw (A) circle (1.5pt);
%\filldraw (B) circle (1.5pt);
%\filldraw (C) circle (1.5pt);
%\filldraw[white] (A) circle [radius=1pt] node {{$$}};
%\filldraw[white] (B) circle [radius=1pt] node {{$$}};
%\filldraw[magenta] (C) circle [radius=1pt] node {{$$}};
\end{example}

 Moreover, this construction naturally induces a map  $$\theta:\Tsex\to K(G_{\alpha/\lambda})^{|E|},$$namely, for every $T\in \Tsex$ we have that $\theta(T)$ is the colouring of $G_{\alpha/\lambda}$ such that $i$ is the entry in cell $c$ in $T$ if and only if $i$ is the colour of the vertex corresponding to $c$ in $G_{\alpha/\lambda}$.
\begin{lemma}\label{lem:theta}
   The map  $\theta:\Tsex\to K(G _{\alpha/\lambda})^{|E|}$ is a bijection.
\end{lemma}
\begin{proof}
    Given a tableau $
    T\in \Tsex$, any directed edge in the corresponding colouring of $G _{\alpha/\lambda}$ will point from a vertex with a smaller label to a vertex with a larger label. \refy{Hence,} every edge in $E$ is an ascent. It follows that $\theta(T)\in K(G _{\alpha/\lambda})^{|E|}$. On the other hand, any colouring of $G_{\alpha/\lambda}$ with $|E|$ ascents must have every edge contributing an ascent, which forces the entries in every row, from left to right, and every column, from bottom to top, to be strictly increasing in the corresponding tableau $S$ of shape $\alpha/\lambda$, thus $S\in \Tsex$. Hence, $\theta$ \finv{is}  \ME{well-defined and a} bijection.
\end{proof}

 This lemma implies that $\bA\ex _{\alpha/\lambda}$ is the coefficient of $q^{|E|}$ in $\svw{X_{G_{\alpha/\lambda}}}(x;q)$ as formal power series in $\QSym$, that is, $$\bA\ex _{\alpha/\lambda}=[q^{|E|}]\svw{X_{G_{\alpha/\lambda}}}(x;q)$$and is the key to our proof.

\begin{theorem}\label{the:symsex}
Let $\alpha$ be a composition and $\lambda$ be a partition, such that $\lambda \subseteq \alpha$. Then
$$\bA\ex _{\alpha/\lambda} \mbox{ is symmetric} \Leftrightarrow \alpha/\lambda = (1^k, |\alpha/\lambda|-k)  \mbox{ for some } 0\leq k \leq \finv{|\alpha / \lambda| -1.}$$
\end{theorem}

\begin{proof} 
The proof for the case $\lambda=\emptyset$ is identical to that of the case $\beta = \emptyset$ for Theorem~\ref{the:symadI}, \JL{except we use Equation~\eqref{eq:sex} instead of Equation~\eqref{eq:sadI}.}

  \refy{Thus,} let us assume that $\lambda\neq \emptyset$. For one direction, if $ \alpha/\lambda = (1^k, |\alpha/\lambda|-k)$ \svw{for some $0\leq k \leq \finv{|\alpha/\lambda|-1}$,} then the symmetry of $\bA\ex _{\alpha/\lambda}$ follows by the first part of this proof. For the other direction, assume $\bA\ex _{\alpha/\lambda}$ is symmetric. Consider the labelled directed acyclic graph $G _{\alpha/\lambda}$. \ME{We consider 2 cases.}
   \begin{enumerate}
       \item   $G _{\alpha/\lambda}$ has exactly one source. It follows by the construction of $G_{\alpha/\lambda}$ that $\ell (\alpha)=\ell(\lambda)$ and  $\lambda_1=\cdots=\lambda_{\ell(\lambda)}=m$ for some $m$. Therefore $\alpha/\lambda$ is the composition $(\alpha_1 - m, \ldots, \alpha_{\ell(\alpha)} -m)$, and must be $(1^k, |\alpha/\lambda|-k)$ \svw{for some $0\leq k \leq \finv{|\alpha/\lambda|-1}$} by the first part of this proof.
       \item  $G _{\alpha/\lambda}$ has more than one source. Recall \svw{that since $\alpha/\lambda$ is connected we have} that $G _{\alpha/\lambda}$ is connected. Then it follows from~\cite[Lemma 4.10]{gps} that there exists some $\gamma, \gamma'\vDash|\alpha/\lambda| $  satisfying $\lambda(\gamma)=\lambda(\gamma')$, such that 
       $$[q^{|E|}M_\gamma]\svw{X_{G_{\alpha/\lambda}}}(x;q) \neq [q^{|E|}M_{\gamma'}]\svw{X_{G_{\alpha/\lambda}}}(x;q)$$that is, the coefficient of $M_\gamma$ and \svw{of} $M_{\gamma '}$ are different in  $[q^{|E|}]\svw{X_{G_{\alpha/\lambda}}}(x;q) = \bA\ex _{\alpha/\lambda}$. Hence, $\bA\ex _{\alpha/\lambda}$ is not symmetric, which contradicts our assumption.  \end{enumerate} \JL{With all cases considered, we are done.} 
\end{proof}

Since \JL{there is only one tableau in ${\svw{\mathcal{T}_{(1^k,  |\alpha/\lambda|-k)}}({\text{cols}<, \text{rows}<})}$ with composition content, namely,} the unique tableau whose cells are filled with $1, \ME{\ldots, k}$ in the leftmost column, and $k+1, \ldots , |\alpha/\lambda|$ in the top row, we get the following.

\begin{corollary}\label{cor:symsex} 
Let $\alpha$ be a composition and $\lambda$ be a partition, such that $\lambda \subseteq \alpha$ and $\bA\ex _{\alpha/\lambda}$ is symmetric. Then
 $$\bA\ex _{\alpha/\lambda} = s_{(1^{ |\alpha/\lambda|})}.$$
\end{corollary}

%%%SUBSec: Weak extendeds
\subsection{Weakly \svw{advanced extended} functions}\label{subsec:wex}
We now come to the final variant.

\begin{definition}\label{def:wextab} Let $\alpha$ be a composition, and $\lambda$ be a partition such that $\lambda \subseteq \alpha$. Then $$\Twex$$is the set of all tableaux of {shape} $\alpha/\lambda$ such that
\begin{enumerate}
\item the column entries weakly increase from  bottom to top; 
\item the row entries weakly increase  from left to right.
\end{enumerate}
\end{definition}

\begin{example}\label{ex:wextab} The following tableau belongs to $\mathcal{T}_{(3,4,4,3,1)/(2,2,1)}({\text{cols}\leq, \text{rows}\leq})$.
$$\begin{ytableau}
3\\
1&4&5\\
\none&3&3&4\\
\none&\none&2&2\\
\none&\none&2
\end{ytableau}$$
\end{example}

Then for a composition $\alpha$ and partition $\lambda$ such that $\lambda \subseteq \alpha$ the \emph{weakly \svw{advanced extended} function} $\cA\ex _{\alpha/\lambda}$ is given by
\begin{equation}\label{eq:wex}
\cA\ex _{\alpha/\lambda} = \sum _{T\in \Twex} \xT = \sum _{\gamma \vDash |\alpha/\lambda|} wd _{(\alpha/\lambda)\gamma} M_\gamma
\end{equation}where $wd _{{(\alpha/\lambda)}\gamma}$ is the number of $T\in \Twex$ with content $\gamma$. The second equality follows by definition.

As given \refy{by Niese et al.} \cite[Figure 1]{NSvWVW2024}, for any composition $\alpha$ we have that
$$\psi(\bA\ex _\alpha) = \cA\ex _\alpha$$and the proof extends to diagrams $\alpha/\lambda$, for $\lambda$ a partition, to give
$$\psi(\bA\ex _{\alpha/\lambda}) = \cA\ex _{\alpha/\lambda}.$$Hence, by Theorem~\ref{the:symsex},  \svw{Corollary~\ref{cor:symsex} and Equation~\eqref{eq:transpose},} we get the following, respectively.

\begin{theorem}\label{the:symwex}
Let $\alpha$ be a composition and $\lambda$ be a partition, such that $\lambda \subseteq \alpha$. Then
$$\cA\ex _{\alpha/\lambda} \mbox{ is symmetric} \Leftrightarrow \alpha/\lambda = (1^k, |\alpha/\lambda|-k)  \mbox{ for some } 0\leq k \leq \finv{|\alpha / \lambda| -1.}$$
\end{theorem}

\begin{corollary}\label{cor:symwex}
Let $\alpha$ be a composition and $\lambda$ be a partition, such that $\lambda \subseteq \alpha$ and $\cA\ex _{\alpha/\lambda}$ is symmetric. Then
 $$\cA\ex _{\alpha/\lambda} = s_{(|\alpha/\lambda|)}.$$
\end{corollary}

%%%SKEW FUNCTIONS
\section{Skew quasisymmetric functions}\label{sec:skew}
In this \refz{section} we prove a general formula for skew functions in $\QSym$, which we then apply in order to justify the restriction \refz{to $\beta$ being a partition} for diagrams $\alpha/\beta$ when classifying when extended Schur functions and their advanced variants are symmetric. For this we will need to define skew functions, and another basis for $\QSym$ consisting of fundamental quasisymmetric functions, in addition to a few algebraic and combinatorial concepts.

Given a composition $\alpha$, the \emph{fundamental quasisymmetric function} $F_\alpha$ is given by
$$F_\alpha = \sum _{\beta \refines \alpha} M_\beta.$$

\begin{example}\label{ex:fbasis} \JL{We have that $F_{(1,2)} = M_{(1,2)}+ M_{(1,1,1)}.$}
\end{example}

For the fundamental quasisymmetric functions, we have that the coproduct in $\QSym$ is given by
\begin{equation}\label{eq:fundcoprod}
\Delta F_\alpha = \sum_{\substack{(\beta,\gamma) \text{ with }\\\beta\cdot \gamma = \alpha \text{ or}\\\beta\odot \gamma=\alpha}} F_\beta\otimes F_\gamma
\end{equation}
where for $\beta = (\beta_1,\ldots, \beta_{\svw{\ell(\beta)}})$ and $\gamma = (\gamma_1,\ldots, \gamma_{\svw{\ell(\gamma)}})$, $\beta\cdot\gamma = (\beta_1,\ldots, \beta_{\svw{\ell(\beta)}},\gamma_1,\ldots, \gamma_{\svw{\ell(\gamma)}})$ is the {\emph{concatenation}} of $\beta$ and $\gamma$,  and $\beta\odot\gamma = (\beta_1, \ldots, \beta_{{\svw{\ell(\beta)}}-1},\beta_{\svw{\ell(\beta)}}+\gamma_1,\gamma_2,\ldots, \gamma_{\svw{\ell(\gamma)}})$ is the {\emph{near-concatenation}} of $\beta$ and $\gamma$.

Before we continue, let $\{\Sk_\alpha\} _{\alpha\vDash n \geq 1}$ be a set of quasisymmetric functions defined as generating functions of tableaux satisfying certain conditions. 

\begin{definition}\label{def:skew} Let $\alpha$ be a composition, \refz{and $\Sk _\alpha$ be the quasisymmetric function defined as a generating function of tableaux of shape $\alpha$ satisfying certain conditions. Let $\Delta$ be the coproduct on $\QSym$.} Then \refz{the} \emph{skew quasisymmetric function} $\Sk_{\alpha/\beta}$ \refz{is the unique quasisymmetric function given} implicitly by
\[{\Delta\Sk_\alpha}= {\sum_{\ME{\beta \subseteq \alpha}} \Sk_\beta\otimes {\Sk_{\alpha/\beta}}}.\]
\end{definition}

We now prove a combinatorial formula for our skew quasisymmetric function. For this we will need the concept of descent set, which comes in four variants given a standard  tableau $T$. 
\begin{enumerate}
    \item {\cite[Definition 3.20]{BBSSZ2014}} The {\emph{dual immaculate descent set}} \svw{is}
    $$\Des_{\dI}(T)=\{i: i+1 \text{ is in a row strictly above $i$ in $T$}\};$$    
    \item {\cite[Defintion 3.4]{NSvWVW2023}} The {\emph{row-strict dual immaculate descent set}} \svw{is}
    $$\Des_{\rdI}(T)=\{i: i+1 \text{ is  in a row weakly below $i$ in $T$}\};$$
    \item {\cite[Definition 9.1]{NSvWVW2024}} The {\emph{$\mathcal{A}^*$-descent set}} \svw{is} $$\Des_{\mathcal{A}^*}(T)=\{i: i+1 \text{ is  in a row strictly below $i$ in $T$}\};$$
    \item {\cite[Definition 9.2]{NSvWVW2024}} The {\emph{$\bA^*$-descent set}} \svw{is} $$\Des_{\bA^*}(T)=\{i: i+1 \text{ is  in a row weakly above $i$ in $T$}\}.$$
\end{enumerate}Also recall that given a set $S=\ME{\{i_1 < i_2 < \cdots < i_s\}} \subseteq \{1, \ldots , n-1\}$, the composition corresponding to it is \ME{$\comp(S) = (i_1, i_2-i_1, \ldots, n-i_s) \vDash n$. }

\begin{example}\label{ex:alldes} Our tableau $T$ below has the following descent sets.
$$T=\tableau{5&8\\3&4&7&9\\1&2&6}$$ 
$$\begin{array}{rlrl}
\Des_{\dI}(T)&=\{ 2,4,6,7\} \quad&\comp(\Des_{\dI}(T)) &= (2,2,2,1,2)\\
\Des_{\rdI}(T)&=\{1, 3,5,8\} \quad&\comp(\Des_{\rdI}(T)) &= (1,2,2,3,1)\\
\Des_{\mathcal{A}^*}(T)&=\{5,8\} \quad&\comp(\JL{\Des_{\mathcal{A}^*}(T)}) &=(5,3,1)\\
\Des_{\bA^*}(T)&= \{1,2,3,4,6,7\} \quad&\comp(\JL{\Des_{\bA^*}(T)}) &= (1,1,1,1,2,1,2)
\end{array}$$
\end{example}

Lastly, we denote by $\JL{\mathcal{ST}} _{\alpha/\beta}\ME{(\mathrm{cols}, \mathrm{rows})}$  the set of all \JL{standard} tableaux of {shape} $\alpha/\beta$ subject to given conditions on the  columns and rows. We are now ready to give our combinatorial formula.

\begin{proposition}[Combinatorial skew formula]\label{prop:skews} Let $\alpha, \beta$ be compositions,   $\Des$ be one of the above descent sets, and $\Sk _\alpha = \sum_{U\in \JL{\mathcal{ST}}_{\alpha}\ME{(\mathrm{cols}, \mathrm{rows})}}F_{\comp(\Des(U))}$. Then
\[\Sk_{\alpha/\beta}=\sum_{T}F_{\comp(\Des(T))}\]
where the sum is over all {standard} tableaux $T\in \JL{\mathcal{ST}} _{\alpha/\beta}\ME{(\mathrm{cols}, \mathrm{rows})}$ for \ME{$\beta \subseteq \alpha$.} Furthermore, if the column restrictions on $\JL{\mathcal{ST}} _{\alpha/\beta}\ME{(\mathrm{cols}, \mathrm{rows})}$ have that certain columns must increase from bottom to top, then \refz{we get the following additional condition. If} a cell in such a column belongs to $\beta$ then every cell below it in that column must also belong to $\beta$. In particular, \refz{this additional condition guarantees that} if every column must increase from bottom to top, then $\beta$ is a partition.
\end{proposition}

\begin{proof}We use the technique of \refy{Bessenrodt et al.} ~\cite[Proposition 3.1]{BLvW}. Let $T$ be a {standard} tableau such that $|T|=n$. For any $k$ with $0\leq k\leq n$, let $\mho_k(T)$ be the standardization of  $T$ \ME{restricted to the} entries $\{n-k+1,{\ldots,} n\}$, namely replace each $\JL{n-k+1} \leq i \leq n$ with $i-(n-k)$. Also let $\Omega_k(T)$ be the remaining cells of $T$ after removing the entries $\{k+1,{\ldots,} n\}$ as in Figure~\ref{fig:omskew}.

\begin{figure}
\[ T=\tableau{5&8\\3&4&7&9\\1&2&6} {\quad \Omega_5(T) = \tableau{5\\3&4\\1&2}\quad \mho_4(T)=\tableau{*&3\\*&*&2&4\\*&*&1} } \]
\caption{An example of {$\Omega_{n-k}(T)$ and $\mho_k(T)$}.}\label{fig:omskew}
\end{figure}

Note that if $T$ is a standard  tableau in $\JL{\mathcal{ST}} _{\alpha}\ME{(\mathrm{cols}, \mathrm{rows})}$, then $T=\Omega_{n-k}(T)\cup(\mho_k(T)+(n-k))$ where $\mho_k(T)+(n-k)$ is $\mho_k(T)$ with $n-k$ added to each entry.  Suppose $\comp(\Des(T)) = \tilde{\alpha}$ with $|\tilde{\alpha}|=n$.  Then we can rewrite~\eqref{eq:fundcoprod} as 
\[\Delta F_{\tilde{\alpha}} = \sum_{i=0}^n \ME{F_{\beta^i}\otimes F_{\gamma^i}}\] 
where \ME{$\beta^i$ and $\gamma ^i$} are the unique compositions such that \ME{$|\beta^i|={n-i}$, $|\gamma^i|={i}$,} and either \ME{$\beta^i\cdot\gamma^i=\tilde{\alpha}$} or \ME{$\beta^i\odot\gamma^i=\tilde{\alpha}$.} {Observe} that \ME{$\beta^i =\comp( \Des({\Omega_{n-i}}(T)))$} and \ME{$\gamma^i = \comp(\Des({\mho_{i}}(T)))$.}

Then
\begin{align*}
\Delta\Sk_\alpha
&=\Delta\left(\sum_{\svw{U}} F_{\comp(\Des({\svw{U}}))}\right)\\
&=\sum_{{\svw{U}}}\Delta F_{\comp(\Des({\svw{U}}))} \\
&=\sum_{{\svw{U}}}\sum^n_{ i=0}\ME{F_{\beta^i}\otimes F_{\gamma^i}}
\end{align*}
where ${\svw{U\in\ }}  \JL{\mathcal{ST}} _{\alpha}\ME{(\mathrm{cols}, \mathrm{rows})}$, \ME{$\beta^i =\comp( \Des({\Omega_{n-i}}(U)))$ and $\gamma^i = \comp(\Des({\mho_{i}}(U)))$.} 

Further, by Definition~\ref{def:skew} we have
\begin{align*}
\Delta \Sk_\alpha &=
\sum_\delta {\Sk_\delta\otimes {\Sk_{\alpha/\delta}}}\\ & = \sum_{\delta}{ \sum_S F_{\comp(\Des(S))}\otimes {\Sk_{\alpha/\delta}}  }
\end{align*} where $S \svw{\ \in\ } \JL{\mathcal{ST}} _{\delta}\ME{(\mathrm{cols}, \mathrm{rows})}$ \svw{for $\delta \subseteq \alpha$.} 
For a fixed $S$ of such a shape $\delta$ with $|\delta|=n-k$ for some $k$, there exists a standard  tableau $T$ in $\JL{\mathcal{ST}} _{\alpha}\ME{(\mathrm{cols}, \mathrm{rows})}$ such that $S=\Omega_{n-k}(T)$. Then $\mho_k(T)$ has shape $\alpha/\delta$.  Similarly, given a standard  tableau $T$ in $\JL{\mathcal{ST}} _{\alpha}\ME{(\mathrm{cols}, \mathrm{rows})}$, we have $T={\Omega_{n-k}(T)\cup(\mho_{k}(T)+(n-k))}$ where $\Omega_{n-k}(T)$ has shape \ME{$\delta \subseteq \alpha$} with $|\delta|=n-k$ and $\mho_{k}(T)$ has shape $\alpha/\delta$. Thus, \refz{since the coproduct expansion of $\Delta \Sk_\alpha$ is unique, we can compare the two right hand sides of the tensor product to obtain} 
\[\Sk_{\alpha/\delta}=\sum_{T}F_{\comp(\Des(T))}\]
where the sum is over all {standard} tableaux $T \ \svw{\in\ } \JL{\mathcal{ST}} _{\alpha/\delta}\ME{(\mathrm{cols}, \mathrm{rows})}$. \ME{Finally} note that by construction, if the column entries of certain columns in every contributing tableau increase from bottom to top because of the given condition, then we must have that if a cell belongs to $\delta$ then every cell below it in that column must also belong to $\delta$. In particular, if every column must increase from bottom to top, then $\delta$ is a partition, and we are done. \end{proof}

\JL{We end by observing that we recover our eight skew functions under consideration in Table~\ref{tab:8functions}.}

\begin{table}
\begin{tabular}[b]{|l|c|c|}\hline
descent set& $\mathcal{ST} _{\alpha/\delta}(\text{cols}, \text{rows})$& skew function\\ \hline
$\Des_{\dI}$&$\mathcal{ST} _{\alpha/\delta}(\text{1st col} <, \text{rows} <)$& $\dI _{\alpha/\delta}$\\
&&\\
&$\mathcal{ST} _{\alpha/\delta}(\text{cols} <, \text{rows} <)$&$\ex _{\alpha/\delta}$\\ \hline
$\Des_{\rdI}$&$\mathcal{ST} _{\alpha/\delta}(\text{1st col} <, \text{rows} <)$&$\rdI _{\alpha/\delta}$\\
&&\\
&$\mathcal{ST} _{\alpha/\delta}(\text{cols} <, \text{rows} <)$&$\rex _{\alpha/\delta}$\\ \hline
$\Des_{\mathcal{A}^*}$&$\mathcal{ST} _{\alpha/\delta}(\text{1st col} <, \text{rows} <)$& $\cA\dI _{\alpha/\delta}$ \\
&&\\
&$\mathcal{ST} _{\alpha/\delta}(\text{cols} <, \text{rows} <)$&$\cA\ex _{\alpha/\delta}$\\ \hline
$\Des_{\bA^*}$&$\mathcal{ST} _{\alpha/\delta}(\text{1st col} <, \text{rows} <)$&$\bA\dI _{\alpha/\delta}$\\
&&\\
&$\mathcal{ST} _{\alpha/\delta}(\text{cols} <,\text{rows} <)$&$\bA\ex _{\alpha/\delta}$\\ \hline
\end{tabular}
\caption{\finv{Eight skew functions recovered.}}\label{tab:8functions}
\end{table}

% Ack
\section{Acknowledgments}\label{sec:ack}  
%%%%%
\refy{The authors would like to thank Shamil Asgarli, Tamsen Whitehead McGinley and Anna Pun for helpful conversations, and the referees for their deep and thoughtful comments.}

%: Bibliography
\bibliographystyle{plain}

\def\cprime{$'$}

\end{document}